\documentclass[12pt,reqno]{amsart}
\usepackage{amsmath, color}
\usepackage[margin=1.5in]{geometry}
\usepackage{graphicx}
\usepackage{amsfonts}
\usepackage{amssymb}
\usepackage[utf8]{inputenc}
\usepackage{comment}
\usepackage{hyperref}
\usepackage{mathtools}

\newtheorem{theorem}{Theorem}[section]
\theoremstyle{plain}

\newtheorem{corollary}{Corollary}[section]

\newtheorem{lemma}{Lemma}[section]

\newtheorem{proposition}{Proposition}[section]

\numberwithin{equation}{section}
\theoremstyle{definition}

\theoremstyle{remark}
\newtheorem{remark}{Remark}[section]
\allowdisplaybreaks

\def\pd{\partial}
\def\re{\mathbb{R}}

\def\mbb{\mathbb}

\newcommand{\eqal}[1]{\begin{equation}\begin{aligned}#1\end{aligned}\end{equation}}

\newcommand{\osc}{\mathrm{osc}}

\newcommand{\Dg}{\Delta_g}
\newcommand{\dg}{\nabla_g}

\title[Singularities of the LMCF]{Singularities of the Lagrangian mean curvature flow at the critical Lagrangian phase
}

\author{Arunima Bhattacharya, Ravi Shankar, Jeremy Wall, and Diego Yepez}

\address{Department of Mathematics, Phillips Hall\\
 the University of North Carolina at Chapel Hill, NC }
\email{arunimab@unc.edu}

\address{Department of Mathematics, Fine Hall\\
Princeton University, Princeton, NJ}
\email{rs1838@princeton.edu}

\address{Department of Mathematics, Phillips Hall\\
 the University of North Carolina at Chapel Hill, NC }
\email{jwall2@unc.edu}

\address{Department of Mathematics, Fine Hall\\
Princeton University, Princeton, NJ}
\email{dy9534@princeton.edu}

\begin{document}

\begin{abstract}
We establish interior estimates for singularities of the Lagrangian mean curvature flow when the Lagrangian phase is critical, i.e., $|\Theta|\geq (n-2)\tfrac{\pi}{2}$, and extend our results to the broader class of Lagrangian mean curvature type equations. Our gradient estimates require certain structural conditions, and we construct $C^{\alpha}$ singular viscosity solutions to show that criticality of the phase is necessary, and that these conditions cannot be removed in dimension one. We also introduce a new method for proving $C^{2,\alpha}$ estimates by exponentiating the arctangent operator into a concave one when $|\Theta|\geq (n-2)\tfrac{\pi}{2}$ and $n>2$.
\end{abstract}

\maketitle

\section{Introduction}
We say that $u$ solves a \textit{Lagrangian mean curvature type equation} in the complex Euclidean space when
\begin{equation}
\sum_{i=1}^n\arctan\lambda_i=\Theta(x,u,Du)\in(-n\frac{\pi}{2},n\frac{\pi}{2}), \label{slag}
\end{equation} where $\lambda_i$'s are the eigenvalues of the Hessian of $u$. The quantity $\Theta$, determined by

\begin{equation}
    \Theta=\sum_{i=1}^n\arctan\lambda_i \label{sl},
\end{equation}
is called the Lagrangian phase or angle of the Lagrangian submanifold $L=(x,Du(x))\subset \mathbb C^n$. By Harvey-Lawson \cite{HL}, the Lagrangian phase determines the mean curvature vector of $L$:
\begin{equation}\label{mc}
    \vec H=J\nabla_g\Theta,
\end{equation}
where $\nabla_g$ is the gradient operator for the induced metric $g_{ij}=\delta_{ij}+u_{ik}\delta^{kl}u_{lk}$ on $L$, and $J$ is the $\frac{\pi}{2}$ rotation matrix in $\mathbb{C}^n$.  When the Lagrangian phase is constant, i.e., $\Theta=c$, one gets the \textit{special Lagrangian equation}. In this case, $H=0$, and $L$ is a volume-minimizing submanifold in $\mathbb C^n$.
A dual form of the special Lagrangian equation is the Monge-Amp\'ere equation $\det D^2u=c$.
This is the potential equation for special Lagrangian submanifolds in $(\mathbb {R}^n\times \mathbb {R}^n, dxdy)$ as interpreted by Hitchin in \cite{Hi}.

We say that $u$ solves the potential equation for prescribed \textit{Lagrangian mean curvature}  when

    \begin{equation*}
   \sum _{i=1}^{n}\arctan \lambda_{i}=\Theta(x). 
\end{equation*}

 A family of Lagrangian submanifolds $X(x,t):\re^n\times\re\to\mbb C^n$ evolves by \textit{Lagrangian mean curvature flow} (LMCF) if it solves
\[(X_t)^\bot=\Delta_gX=\vec H.
\]
After a change of coordinates, one can locally write $X(x,t)=(x,Du(x,t))$, such that $\Delta_gX=(J\bar\nabla\Theta(x,t))^\bot$.  This means a local potential $u(x,t)$ evolves by the parabolic equation
\begin{align*}
    u_t&=\sum_{i=1}^n\arctan\lambda_i,\\
    &u(x,0):=u(x).
\end{align*}

Symmetry reductions of the Lagrangian mean curvature flow reduce the above local parabolic representation to an elliptic equation for $u(x)$, which models singularities of the mean curvature flow (see Chau-Chen-He \cite{CCH}). If $u(x)$ solves
\eqal{
\label{s}
\sum_{i=1}^n\arctan\lambda_i=s_1+s_2(x\cdot Du(x)-2u(x)),
}
then $X(x,t)=\sqrt{1-2s_2t}\,(x,Du(x))$ is a \textit{shrinker} or \textit{expander} solution of the LMCF, if $s_2>0$ or $s_2<0$ respectively.  If $u(x)$ solves
\eqal{
\label{tran}
\sum_{i=1}^n\arctan\lambda_i=\gamma_1+\gamma_2\cdot x+\gamma_3\cdot Du(x),
}
then $X(x,t)=(x,Du(x))+t(-\gamma_3,\gamma_2)$ is a \textit{translator} solution of the LMCF. If $u(x)$ solves
\eqal{
\label{rotator}
\sum_{i=1}^n\arctan\lambda_i=r_1+\frac{r_2}{2}(|x|^2+|Du(x)|^2),
}
then $X(x,t)=\exp(r_2tJ)(x,Du(x))$ is a \textit{rotator} solution of the LMCF.

\subsection*{Notation }Before we present our main results, we clarify some terminology.
\begin{itemize}
\item[I.] By $B_R$ we denote a ball of radius $R$ centered at the origin.
\item[II.] We denote the oscillation of $u$ in $B_R$ by $\osc_{B_R}(u)$.
\item[III.] Let $\Gamma_R = B_R\times u(B_R)\times Du(B_R)\subset B_R\times\re\times\re^n$. Let $\nu_1,\nu_2$ be constants such that for $\Theta(x,z,p)$, we assume the following structure conditions
\begin{align}
    |\Theta_x|,|\Theta_z|,|\Theta_p|&\leq \nu_1(\osc_{B_R} u, \|Du\|_{L^\infty(B_R)}),\label{struct}\\
    |\Theta_{xx}|,|\Theta_{xz}|,|\Theta_{xp}|,|\Theta_{zz}|,|\Theta_{zp}| &\leq \nu_2(\osc_{B_R} u, \|Du\|_{L^\infty(B_R)}), \nonumber
\end{align}
for all $(x,z,p)\in\Gamma_R$. In the above partial derivatives, the variables $x,z,p$ are treated as independent of each other. Observe that this indicates that the above partial derivatives do not have any $D^2u$ or $D^3u$ terms.
 \end{itemize}
\medskip

Our main results are the following.

\begin{theorem}[Hessian Estimates]\label{main1}
If $u$ is a smooth solution of \eqref{slag} on $B_R\subset\re^n$, with $n\geq 3$ and $|\Theta|\geq (n-2)\frac{\pi}{2}$ where $\Theta(x,z,p)\in C^2(\Gamma_R)$ satisfies \eqref{struct} and is partially convex in $p$, then we get
\[
|D^2u(0)|\leq C_1\exp{[C_2\|Du\|_{L^\infty(B_{R})}^{2n+3}/R^{2n+3}]}
\]
where $C_1$ and $C_2$ are positive constants depending on $\nu_1,\nu_2$, and $n$.
\end{theorem}

\begin{remark}\label{sing}
    From the singular solutions constructed in \cite[(1.13)]{BS2}, it is clear that the Hessian estimates stated in Theorem \ref{main1} fail to hold in the absence of partial convexity of $\Theta$ with respect to the gradient variable $Du$.
\end{remark}

\begin{corollary}\label{main0}
If $u$ is a smooth solution of any of these equations: \eqref{s}, \eqref{tran}, and \eqref{rotator} on $B_{R}\subset \mathbb{R}^{n}$ where  $n\geq 3$ and $|\Theta|\geq (n-2)\frac{\pi}{2}$, then we obtain 
\[
|D^2u(0)|\leq C_1\exp{[C_2\|Du\|_{L^\infty(B_{R})}^{2n+5}/R^{2n+5}]}
\]
where $C_1$ and $C_2$ are positive constants depending on $n$ and the following: $s_2$ for \eqref{s}, $\gamma_2,\gamma_3$ for \eqref{tran}, and $r_2$ for \eqref{rotator}.
\end{corollary}
\begin{remark}
    In the case of \eqref{tran}, we get a better estimate
    \[
    |D^2u(0)|\leq C_1\exp{[C_2\|Du\|_{L^\infty(B_{R})}^{2n+3}/R^{2n+3}]}.
    \]
\end{remark}

Next, our gradient estimates are as follows. 

\begin{theorem}[Gradient Estimates]\label{grad}
    Let $u$ be a smooth solution of \eqref{slag} on $B_1\subset\re^n$ with $|\Theta|\geq (n-2)\frac{\pi}{2}$. Suppose that $\Theta(x,z,p)$ satisfies the following conditions:
    \begin{itemize}
        \item[(a)] $\Theta(x,z,p)$ is uniformly $C^2$ in $x$, i.e., $|\Theta_{xx}|\leq \|\Theta_{xx}\|_{L^\infty(\Gamma_1)}$ independently of $u$ and $Du$, and that \begin{equation}\label{eq:interp}
        \lvert D_x\Theta(x,z,p)\rvert \le C\Bigl(\Theta(x,z,p)- \tfrac{(n-2)\pi}{2}\Bigr)^{1/2}.
        \end{equation}
        \item [(b)]$\Theta_z$ is non-negative: \begin{equation}
        D_z\Theta(x,z,p) \;=\; \Theta_z(x,z,p)\;\ge 0. \label{eq:monotone-u}
        \end{equation}
        \item [(c)]$\Theta$ has superlinear decay $o(1/\lvert p\rvert)$ as $p\to\infty$ and $\Theta\to(n-2)\pi/2$, more quantitatively
\begin{equation}\label{eq:decay-p}
\lvert p\rvert\,\lvert D_p\Theta(x,z,p)\rvert \;\le\; C\Bigl(\Theta(x,z,p)- \tfrac{(n-2)\pi}{2}\Bigr).
\end{equation}
    \end{itemize}
    Then $u$ satisfies the following estimate: 
    \begin{equation*}
    |Du(0)|\leq C(n,\|\Theta_{xx}\|_{L^\infty(\Gamma_1)})\left(1+(\text{osc}_{B_1}u)^2\right).
\end{equation*}
\end{theorem}
\medskip

\begin{remark}
    In Section \ref{Sec_sing_soln}, we construct singular $C^{\alpha}$ viscosity solutions to show that none of the conditions (a), (b), and (c) can be removed in dimension one.
\end{remark}

For solutions of the special Lagrangian equation with critical and supercritical phase $|\Theta|\geq (n-2)\frac{\pi}{2}$, Hessian estimates were obtained by Warren-Yuan \cite{WY9,WY}, Wang-Yuan \cite{WaY}, Li \cite{Lcomp} via a compactness approach, Shankar \cite{shankar2024hessian} via a doubling approach, and Zhou \cite{ZhouHess} for equations requiring Hessian constraints which generalize criticality. The singular $C^{1,\alpha}$ solutions to \eqref{sl} constructed by Nadirashvili-Vl\u{a}du\c{t} \cite{NV} and Wang-Yuan \cite{WdY} show that interior regularity is not possible for subcritical phases $|\Theta|<(n-2)\frac{\pi}{2}$, without an additional convexity condition, as shown in Bao-Chen \cite{BCconvex}, Chen-Warren-Yuan \cite{CWY}, and Chen-Shankar-Yuan \cite{CSY}, and that the Dirichlet problem is not classically solvable for arbitrary smooth boundary data. Recently, viscosity solutions to \eqref{sl} that are Lipschitz but not $C^1$ were constructed by Mooney-Savin \cite{MooneySavin}.

For solutions of the Lagrangian mean curvature equation, Hessian estimates for convex smooth solutions with $C^{1,1}$ phase were obtained by Warren in \cite[Theorem 8]{WTh}. For $C^{1,1}$ critical and supercritical phase, interior Hessian and gradient estimates were established by Bhattacharya \cite{AB, AB2d} and Bhattacharya-Mooney-Shankar \cite{BMS} (for $C^2$ phase), respectively. See also Lu \cite{Siyuan}. Interior Hessian estimates for supercritical $C^{0,1}$ phase were derived by Zhou \cite{Zhou1}. Recently, interior Hessian estimates for critical and supercritical $C^{0,1}$ phases were proven by Ding \cite{Ding1}. For convex viscosity solutions, interior regularity was established for $C^2$ phase by Bhattacharya-Shankar in \cite{BS1} and optimal regularity for the broader class of equations, given by \eqref{slag} was derived in \cite{BS2}. If $\Theta$ is merely in $C^{\alpha}$ and supercritical, $|\Theta|\geq (n-2)\frac{\pi}{2}+\delta$, counterexamples to Hessian estimates exist as shown in \cite{AB1}. For more work on the special Lagrangian and Lagrangian mean curvature type equations, we refer the reader to \cite{CaoWang,WHB24,QZ24,LiuBao}.

Bhattacharya-Wall obtained a priori Hessian bounds for singularities of the LMCF and for the broader class of equations \eqref{slag}: first in the case of convex potentials \cite{BWall1}, and later for supercritical phases \cite{BWall2}. In the supercritical range, one has $D^2u \geq -\cot\delta I_n$, which permits a Lewy-Yuan rotation \cite[p.122]{YY}. The rotation yields a uniformly elliptic Jacobi inequality on the rotated graph, akin to \cite{CWY}, enabling the application of the local maximum principle \cite[Theorem 9.20]{GT}. At the critical phase ($\delta=0$), however, this rotation technique fails. 

\medskip

Let us give a brief overview of the methods used in our proofs. To prove gradient estimates, we follow the approach of \cite{BMS}. Since the phase depends on $(x, u(x), Du(x))$, the appearance of terms involving the gradient of the phase prevents direct application of earlier methods. We treat each gradient contribution separately and derive structural conditions under which a gradient estimate holds. We further construct $C^\alpha$ singular viscosity solutions demonstrating that none of the conditions can be removed in dimension one, and that the criticality condition on the phase is necessary in higher dimensions.

It is interesting to note that despite the fact that the Lagrangian angles of these solutions can be interpreted as smooth functions on the ambient space $\re^n\times\re^n$, these viscosity solutions have no geometrical interpretation as Lagrangian submanifolds, and the mean curvature is not bounded.  This is in contrast to the $C^{1,1}$ Lagrangian submanifolds of the Lipschitz viscosity solutions in \cite{MooneySavin}, and the analytic submanifolds of the $C^{1,\alpha}$ viscosity solutions in \cite{NV,WdY}.  There are non-Dini viscosity solutions for the complex Monge-Amp\`ere equation with constant coefficients in \cite{WW}.

The main difficulty in deriving Hessian estimates for \eqref{slag} arises from the use of the Michael-Simon Sobolev inequality \cite[Theorem 2.1]{MS}. Because $\Theta$ depends on the gradient, substituting the mean curvature term into the Sobolev inequality introduces additional Hessian-dependent terms. To control these, we use an idea of \cite{Ding1}, which removes the Hessian dependence at the cost of introducing higher-order dependence on the gradient.

Once the Hessian is bounded, the main obstruction to obtaining $C^{2,\alpha}$ estimates arises from the lack of concavity of the arctangent operator and the dependence of the phase on the potential and its gradient. To overcome this, we develop a new technique that exponentiates the arctangent operator at the critical value, exploiting the boundedness of the Hessian and the assumption that $n \geq 3$. Note that this approach had previously been available only for the supercritical phase, where the exponentiation method of \cite{CPW, CW2} relied on $\delta > 0$, and no such method was known at the critical phase ($\delta = 0$). While applicable in higher dimensions $n\ge 3$, Yuan’s counterexamples \cite{Ycount_notes} show that the technique fails in dimension two.  In this situation, an upwards rotation employed in Bhattacharya-Ogden \cite{BOHamstat} can yield the estimate. We conclude by presenting alternative approaches to derive the $C^{2,\alpha}$ estimate in Remark \ref{C2a}.

\subsection*{Organization} The paper is organized as follows. In Section \ref{sec-prel}, we introduce notation and recall some past results. In Section \ref{sec-grad}, we prove gradient estimates, followed by Section \ref{Sec_sing_soln}, where we construct singular $C^\alpha$ viscosity solutions to prove the necessity of the conditions imposed to prove the gradient estimates. In Section \ref{sec-Hess}, we prove Hessian estimates, and lastly in Section \ref{sec-c2alpha}, we prove $C^{2,\alpha}$ estimates.

\subsection*{Acknowledgments} AB is grateful to Yu Yuan for pointing out that the exponentiation technique described in Lemma \ref{Expo} fails in two dimensions.
AB acknowledges
the support of NSF grant DMS-2350290, the Simons Foundation grant MPS-TSM-00002933, and a Bill Guthridge fellowship from UNC-Chapel Hill. JW acknowledges
the support of NSF RTG grant DMS-2135998 and NSF grant DMS-2350290.

\section{Preliminaries}\label{sec-prel}

For the convenience of the readers, we recall some preliminary results. We first introduce some notations that will be used in this paper.
The induced Riemannian metric on the Lagrangian submanifold $L=(x,Du(x))\subset \mathbb{R}^n\times\mathbb{R}^n$ is given by
\[g=I_n+(D^2u)^2 .
\]
We denote
 \begin{align*} 
    \partial_i=\frac{\partial}{\partial x_i} \text{ , }
     \partial_{ij}=\frac{\partial^2}{\partial x_i\partial x_j} \text{ , }
     u_i=\partial_iu \text{ , }
    u_{ij}=\partial_{ij}u.
    \end{align*}
  Note that for the functions defined below, the subscripts on the left do not represent partial derivatives\begin{align*}
    h_{ijk}=\sqrt{g^{ii}}\sqrt{g^{jj}}\sqrt{g^{kk}}u_{ijk},\quad
    g^{ii}=\frac{1}{1+\lambda_i^2}.
    \end{align*}
Here $(g^{ij})$ is the inverse of the matrix $g$ and $h_{ijk}$ denotes the second fundamental form when the Hessian of $u$ is diagonalized.
The volume form, gradient, and inner product with respect to the metric $g$ are given by
\begin{align*}
    dv_g=\sqrt{\det g}dx &= Vdx \text{ , }\qquad
    \nabla_g v=g^{ij}v_iL_j,\\
    \langle\nabla_gv,\nabla_g w\rangle_g &=g^{ij}v_iw_j \text{ , }\quad
    |\nabla_gv|^2=\langle\nabla_gv,\nabla_g v\rangle_g.
\end{align*}

We recall the following lemma. 
\begin{lemma}\cite[Lemma 2.2]{WaY}\label{y1}
		Suppose that the ordered real numbers $\lambda_{1}\geq \lambda_{2}\geq...\geq \lambda_{n}$ satisfy \eqref{sl} with $\Theta\geq (n-2)\frac{\pi}{2}$. 
		Then we get \begin{enumerate}
			\item $\lambda_{1}\geq \lambda_{2}\geq...\geq \lambda_{n-1}>0,\quad \lambda_{n-1}\geq |\lambda_{n}|$.
			\item $\lambda_{1}+(n-1)\lambda_{n}\geq 0$.
			\item $\sigma_{k}(\lambda_{1},...,\lambda_{n})\geq 0$ for all $1\leq k\leq n-1$ and $n\geq 2$.
		\end{enumerate}
	\end{lemma}

\subsection{Jacobi Inequality}
We recall the following Jacobi-type inequality for the slope of the gradient graph $(x,Du(x))$ from \cite{BWall2}. 
\begin{lemma}\cite[Lemma 3.1]{BWall2}\label{ptJ}
    Let $u$ be a smooth solution of \eqref{slag} in $\re^n$ with $n\geq 3$ and $\Theta \geq (n-2)\frac{\pi}{2}$. Suppose that the ordered eigenvalues $\lambda_1\geq \lambda_2 \geq \cdots \geq \lambda_n$ of the Hessian $D^2u$ satisfy $\lambda_1 = \cdots =  \lambda_m > \lambda_{m+1}$ at a point $x_0$. Then the function $ b_m = \frac{1}{m}\sum_{i = 1}^m \ln\sqrt{1 + \lambda_i^2}$ is smooth near $x_0$ and at $x_0$ satisfies
    \begin{equation}\label{ptJI}
        \Dg b_m \geq c(n)|\dg b_m|^2 - C(\nu_1,\nu_2,n)(1 + |Du(x_0)|^2 ).
    \end{equation}
\end{lemma}

        Note that for the phase $\Theta(x,u,Du)$, assuming the Hessian $D^2u$ is diagonalized at a point $x_0$, we get
\begin{align}
    \pd_i \Theta(x,u,Du) &= \Theta_{x_i} + \Theta_z u_i + \sum_k \Theta_{p_k}u_{ki} \label{dipsi}\\
    &\overset{x_0}{=} \Theta_{x_i} + \Theta_z u_i + \Theta_{p_i}\lambda_i. \nonumber
\end{align}
        Taking the $j$-th partial derivative of \eqref{dipsi} we get 
\begin{align*}
    \pd_{ij}\Theta(x,u,Du) &= \Theta_{x_ix_j} + \Theta_{x_i z}u_j + \sum_{r=1}^n \Theta_{x_ip_r}u_{rj}\nonumber\\
    & \qquad +\left(\Theta_{zx_j} + \Theta_{zz}u_j + \sum_{s=1}^n \Theta_{z p_s}u_{sj} \right)u_i + \Theta_z u_{ij}\nonumber\\
    & \qquad +\sum_{k=1}^n \left(\Theta_{p_kx_j} + \Theta_{p_kz}u_j + \sum_{t=1}^n \Theta_{p_kp_t}u_{tj}\right)u_{ki}+\sum_{k=1}^n \Theta_{p_k}u_{kij}\nonumber\\
    & \overset{x_0}{=} \Theta_{x_ix_j} + \Theta_{x_i z}u_j + \Theta_{x_ip_j}\lambda_j\label{dijpsi@p}\\
    & \qquad +\left(\Theta_{zx_j} + \Theta_{zz}u_j + \Theta_{z p_j}\lambda_j \right)u_i + \Theta_z \lambda_i\delta_{ij}\nonumber\\
    & \qquad +\left(\Theta_{p_ix_j} + \Theta_{p_iz}u_j + \Theta_{p_ip_j}\lambda_j\right)\lambda_i + \sum_{k=1}^n \Theta_{p_k}u_{kij}.\nonumber
\end{align*}

We obtain the following by combining Corollaries 3.1, 3.2, and 3.3 of \cite{BWall2}. 
\begin{corollary}\label{singcor}
    Let $u$ be a smooth solution of \eqref{s}, \eqref{tran}, or \eqref{rotator} in $B_1\subset\re^n$ where $|\Theta|\geq (n-2)\frac{\pi}{2}$ and $n\geq 3$. Assuming the Hessian $D^2u$ is diagonalized at $x_0\in B_1$, \eqref{ptJI} holds with constant $C(n,s_2)(1+|Du(x_0)|^2)$ for \eqref{s}, $C(n,\gamma_2,\gamma_3)$ for \eqref{tran}, or $=C(n,r_2)(1 + |Du(x_0)|^2)$ for \eqref{rotator}.
\end{corollary}

Next, we state the integral version of the Jacobi-type inequality from \cite{BWall2}.
\begin{proposition}\cite[Proposition 3.1]{BWall2}\label{IntJI}
    Let $u$ be a smooth solution of \eqref{slag} in $\re^n$ with $n\geq 3$ and $|\Theta|\geq (n-2)\frac{\pi}{2}$. Let
    \begin{equation*}
        b=b_1 = \log\sqrt{1 + \lambda_{\max{}}^2}
    \end{equation*}
    where $\lambda_{\max{}}$ is the largest eigenvalue of $D^2u$, i.e., $\lambda_{\max{}} = \lambda_1 \geq \lambda_2\geq\cdots\geq\lambda_n$. Then, for all non-negative $\phi\in C_0^\infty(B_R)$, $b$ satisfies the integral Jacobi inequality
    \begin{equation*}
        \int_{B_R}-\langle\nabla_g \phi, \nabla_g b\rangle_g\;dv_g \geq c(n) \int_{B_R} \phi|\nabla_g b|^2 dv_g - \int_{B_R}C(n,\nu_1,\nu_2)(1 + |Du(x)|^2)\phi dv_g.
    \end{equation*}
\end{proposition}
\begin{remark}
    This holds via the same proof for solutions to \eqref{s},\eqref{tran},\eqref{rotator} each with their respective $C$.
\end{remark}

\section{Proof of the Gradient Estimates}\label{sec-grad}
\begin{proof} The method of this proof is inspired by \cite{BMS}. We denote $ca\leq b\leq Ca$ by $a\sim b$. Let $M = \osc_{B_1}u$ and replace $u$ with $u - \inf_{B_1}u + M$ so that $M\leq u \leq 2M$. We will use the test function $w = \eta|Du| + Au^2/2$ where $\eta = 1 -|x|^2$ and $A = 3\sqrt{n}/M$. Suppose $w$ achieves its maximum at $x_0\in \bar{B}_1$. If $x_0\in\pd B_1$, then we are done so suppose $x_0\in B_1$. Diagonalize $D^2u(x_0) = \text{diag}(\lambda_1,\dots,\lambda_n)$. We will assume $u_n \geq |Du|/\sqrt{n} > 0$ and observe that for each $k$, at the maximum point $x_0$,
\[
0 = \pd_k w(x_0) = \eta\frac{u_k\lambda_k}{|Du|} + \eta_k|Du| + Auu_k.
\]
By taking $A = 3\sqrt{n}/M$, we see that
\begin{equation}\label{dwzero}
    \eta\lambda_n \frac{u_n}{|Du|} \in (-c(n),-C(n))|Du|.
\end{equation}
Thus $\lambda_n < 0$, and
\[
\eta \sim \frac{|Du|}{|\lambda_n|}.
\]
Since $|Du|\lesssim \eta|\lambda_n|$ we may assume $|\lambda_n|> 1$ as otherwise we would be done. 

At $x_0$, we also have
\begin{align}
    0&\geq g^{ij}\pd_{ij}w(x_0)\nonumber\\
    &= |Du|g^{ij}\eta_{ij}+ 2g^{ij}\eta_i|Du|_j + \eta g^{ij}|Du|_{ij} + Aug^{ij}u_{ij} + Ag^{ij}u_iu_j.\label{d2w}
\end{align}
As in \cite{BMS}, we get 
\begin{align}
    |Du|g^{ij}\eta_{ij} &= -2\sum_{j=1}^n\frac{1}{1+\lambda_j^2}|Du| \gtrsim -\frac{1}{\lambda_n^2}|Du| \sim \frac{\eta}{|\lambda_n|},\label{c1}\\
    2g^{ij}\eta_i|Du|_j &\overset{x_0}{\gtrsim} -\sum_{j=1}^n\frac{1}{1+\lambda_j^2}\frac{u_j\lambda_j}{|Du|}\gtrsim -\frac{1}{|\lambda_n|},\label{c2}\\
    A u g^{ij}u_{ij} &= Au\sum_{j=1}^n\frac{\lambda_j}{1+\lambda_j^2} \gtrsim - \frac{1}{|\lambda_n|},\label{c3}\\
    A g^{ij}u_iu_j &\geq A\frac{u_n^2}{1+\lambda_n^2} \sim A \frac{|Du|^2}{\lambda_n^2} \sim A\frac{\eta|Du|}{|\lambda_n|}\label{goodt}.
\end{align}
The remaining term is 
\begin{align}
    \eta g^{ij}|Du|_{ij} &= \eta g^{ij} \frac{u_{ijk}u_k}{|Du|} + \eta \sum_{j=1}^n g^{jj}\frac{(|Du|^2 - u_j^2)\lambda_j}{|Du|^3}\nonumber\\
    &\geq \eta \frac{u_k}{|Du|}\pd_k\Theta(x,u(x),Du(x))\nonumber\\
    &= \eta\frac{u_k}{|Du|}\left(\Theta_{x_k}(x,u,Du) + \Theta_{z}(x,u,Du) u_k + \Theta_{p_k}(x,u,Du) \lambda_k\right).\label{dkThet}
\end{align}
We note that by the argument in \cite{BMS}, we get that
\begin{equation}
    \Theta - \frac{(n-2)\pi}{2} \leq \frac{1}{|\lambda_n|}.
\end{equation}

By the structure condition \eqref{struct} on $\Theta$, we will here assume $|\Theta_{xx}|\leq \|\Theta_{xx}\|_{L^\infty(\Gamma_1)}$ where $\|\Theta_{xx}\|_{L^\infty(\Gamma_1)}$ is assumed independent of $u$ and $Du$.
If we also assume \eqref{eq:interp}, i.e.,
\[
|\Theta_x| \lesssim_{\Theta_{xx}} \left(\Theta - \frac{(n-2)\pi}{2}\right)^\frac{1}{2},
\]
which is similar to the interpolation inequality \cite[Lemma 7.7.2]{lhorm1}, then 
\begin{equation}\label{thetx}
    \eta \frac{u_k}{|Du|}\Theta_{x_k} \gtrsim_{\Theta_{xx}} -\eta |\lambda_n|^{-1/2}.
\end{equation}

To bound the middle term of \eqref{dkThet}, we assume \eqref{eq:monotone-u}, that is, $\Theta_z \geq 0$.  Note that such a condition plays a role in the $C^0$ maximum principle used in existence and uniqueness theory.
Hence
\begin{equation}\label{thetu}
    \eta \frac{u_k^2}{|Du|}\Theta_z = \eta|Du|\Theta_z \geq 0.
\end{equation}

To bound the final term of \eqref{dkThet}, we will assume \eqref{eq:decay-p}, namely, 
\[
\lvert p\rvert\,\lvert D_p\Theta(x,z,p)\rvert \;\le\; C\Bigl(\Theta(x,z,p)- \tfrac{(n-2)\pi}{2}\Bigr).
\]
Then, using \eqref{dwzero} we see that
\begin{align}
    \eta\frac{u_k\lambda_k}{|Du|}\Theta_{p_k} &= -\Theta_{p_k}(\eta_k|Du| + A u u_k)\nonumber\\
    &\geq -C(n)|Du\|D_p\Theta|\nonumber\\
    &\geq -C(\Theta(x,z,p)- \tfrac{(n-2)\pi}{2}) \nonumber\\
    &\geq -C|\lambda_n|^{-1} \label{thetp}.
\end{align}

Rearranging the terms of \eqref{d2w}, using all of the above estimates \eqref{c1}, \eqref{c2}, \eqref{c3}, \eqref{goodt}, and inserting \eqref{thetx}, \eqref{thetu}, and \eqref{thetp} into \eqref{dkThet}, we get
\[
\eta |Du|\lesssim_{\Theta_{xx}} M(1 + \eta|\lambda_n|^{1/2}) \lesssim M(1 + (\eta|Du|)^{1/2}).
\]
By Young's inequality, we then get
\[
\eta|Du(x_0)| \leq C(\|\Theta_{xx}\|_{L^\infty(\Gamma_1)},n)( 1 + (\osc_{B_1}u)^2)
\]
which gives
\[
|Du(0)| \leq w(0) \leq w(x_0) \leq C(\|\Theta_{xx}\|_{L^\infty(\Gamma_1)},n)( 1 + (\osc_{B_1}u)^2),
\]
concluding the proof. \end{proof}

\section{Singular Non-Lipschitz Viscosity Solutions}\label{Sec_sing_soln}
Using counterexamples, we show that $\Theta\ge (n-2)\pi/2$ is necessary for Lipschitz regularity, and that none of the conditions in Theorem \ref{grad} can be removed in dimension $n=1$.  In dimension $n=1$, we construct H\"older-but-not-Lipschitz viscosity solutions for $C^2$ phases which fail at least one of each of the conditions \eqref{eq:interp}, \eqref{eq:monotone-u}, and \eqref{eq:decay-p}.  In higher dimensions, the quadratic lifts of these $n=1$ solutions provide non-Lipschitz viscosity solutions with subcritical phase.

\smallskip
We note that H\"older viscosity solutions are not known for the constant phase special Lagrangian equation, or even for $F(D^2u)=\Theta(x)$ with $\Theta\in C^{1,\alpha}$.  The reason is that the even function $u=|x_1|^\alpha$ is not a supersolution.  More generally, homogeneous functions with sign-constant determinant will not be admissible \cite{mooney2022homogeneous}.

\smallskip
If $Du$ dependence in $\Theta$ is allowed, then we can instead look for odd functions $u=\text{sign}(x_1)|x_1|^\alpha$.  If we identify this as a viscosity solution of $F(D^2u)=\Theta(x)$, then $\Theta(x)$ must be discontinuous at $x_1=0$.  But since $Du$ is itself singular, there are smooth combinations of $x_1,u,$ and $Du$ which agree with $D^2u$.

\subsection{Fails condition \texorpdfstring{\eqref{eq:monotone-u}}{(0.3)}}\label{subsec:fail-03}
Let $u(x_1) = \operatorname{sign}(x_1)\,\lvert x_1\rvert^\alpha$ for $0<\alpha<1$. Then
\begin{align*}
u'' &= \alpha(\alpha-1)\,\operatorname{sign}(x_1)\,\lvert x_1\rvert^{\alpha-2},\\
u'  &= \alpha\,\lvert x_1\rvert^{\alpha-1}.
\end{align*}
Hence
\begin{equation*}
u\,(u')^q \;=\; \alpha^q\,\operatorname{sign}(x_1)\,\lvert x_1\rvert^{\alpha(1+q)-q}
\;=\; \alpha^{q-1}(\alpha-1)^{-1}\,u''
\end{equation*}
provided $\alpha(1+q)-q = \alpha-2$, i.e., $q\alpha = q-2$, so $\alpha = (q-2)/q$. Therefore $u$ is a viscosity solution of
\begin{equation*}
\arctan(u'') \;=\; -\,\arctan\!\bigl((1-\alpha)\,\alpha^{1-q}\,u\,(u')^q\bigr).
\end{equation*}
This clearly fails condition \eqref{eq:monotone-u}, due to the sign of $\Theta_u$.

More generally, for integers $m\ge 0$,
\begin{equation*}
u^{2m+1}\,(u')^q \;=\; C\,\operatorname{sign}(x_1)\,\lvert x_1\rvert^{\alpha(2m+1+q)-q}
 \;=\; -C_1\,u''
\end{equation*}
if $\alpha-2=\alpha(2m+1+q)-q$, i.e., $\alpha=(q-2)/(2m+q)$. In this case,
\begin{equation*}
\arctan(u'') \;=\; -\,\arctan\!\bigl(C_1^{-1}\,u^{2m+1}\,(u')^q\bigr).
\end{equation*}

\subsection{Fails condition \texorpdfstring{\eqref{eq:interp}}{(0.2)}}
\label{sec:xcounter}
For integers $m\ge 0$,
\begin{equation*}
x_1^{2m+1}\,(u')^q \;=\; \operatorname{sign}(x_1)\,\lvert x_1\rvert^{2m+1+q\alpha-q}
 \;=\; C\,u''
\end{equation*}
if $2m+1+q\alpha-q = \alpha-2$, i.e., $\alpha=(q-3-2m)/(q-1)$. In this case,
\begin{equation*}
\arctan(u'') \;=\; -\,\arctan\!\bigl(C^{-1}\,x_1^{2m+1}\,(u')^q\bigr).
\end{equation*}
This fails the uniform $C^2$ bound in $x$, since taking two derivatives yields additional chain rule
factors of $(u')^q$, which range over the entire real line. Hence the resulting function cannot be
bounded in the variable $x_1^{2m+1}(u')^q$.

\subsection{Fails conditions \texorpdfstring{\eqref{eq:decay-p}}{(0.4)} and \texorpdfstring{\eqref{eq:interp}}{(0.2)}}
Let $L=-\log\lvert x_1\rvert$ be a positive even function near $x_1=0$, and set $u'(x_1)=L^{1/2}$ with
\[
u(x_1)\sim x_1 L^{1/2} + o\bigl(x_1 L^{1/2}\bigr).
\]
Then
\begin{equation*}
2u''(x_1) \;=\; -\,L^{-1/2}\,x_1^{-1} \;=\; -x_1\,L^{-1/2}\,x_1^{-2}
 \;=\; -\,x_1\,e^{2(u')^2}/u'.
\end{equation*}
Therefore, near $x_1=0$, $u$ is a viscosity solution of an equation of the form
\begin{equation*}
\arctan u'' \;=\; -\,\arctan\!\left[\tfrac12\,x_1\,g(u')\right],\qquad
g(u') = e^{2(u')^2}/u' \quad \text{near } \lvert u'\rvert=\infty.
\end{equation*}
This fails condition \eqref{eq:decay-p} due to Gaussian growth, and it also fails condition \eqref{eq:interp}.

\subsection{Higher dimensions}
Let $v(x_1)$ be the non-Lipschitz viscosity solution of Section \ref{sec:xcounter} constructed to $\arctan v''=\Theta_1(x_1,u_1)\in[-\pi/2,\pi/2]$.  For $a_i\in\re$, we define a solution in separated variables
$$
u(x_1,\dots,x_n)=v(x_1)+\sum_{i=2}^n a_ix_i^2/2.
$$
Then $u$ is a non-Lipschitz viscosity solution of
\eqal{
F(D^2u)&=\Theta_1(x_1,u_1)+\sum_{i=2}^n\arctan a_i=:\Theta(x,Du).
}
Near $\Theta_1=-\pi/2$, it is not possible to choose $a_i$ so that $\Theta\ge (n-2)\pi/2$, so the total phase $\Theta$ attains subcritical values.  This shows that $\Theta(x,u,Du)\ge (n-2)\pi/2$ cannot be dropped from the gradient estimate.

\subsection{Remarks}
It is an interesting question to find a solution that fails \eqref{eq:decay-p} alone, but it must be noted that \eqref{eq:decay-p} cannot be dropped.  It is also interesting to see whether any of these conditions can be dropped in higher dimensions, provided that $\Theta\ge (n-2)\pi/2$.  The $n=1$ lifts drop below the critical value $(n-2)\pi/2$ near the singularities.

\smallskip
From a geometric point of view, if $\Theta=\Theta(x,Du)$, then the mean curvature is the normal projection $H=(J\nabla_{(x,y)}\Theta(x,y))^\bot$ of the smooth ambient gradient $\nabla\Theta=(\Theta_x,\Theta_y)$ restricted to $(x,Du)$.  Note that the mean curvature is not bounded for these examples near the singularity.

\smallskip
Our examples reflect that smooth such data $\Theta(x,y)$ for the mean curvature of a Lagrangian submanifold is not sufficient to make the submanifold $(x,Du)$ continuous, or even bounded.  Note that although the viscosity solutions in \cite{MooneySavin} are Lipschitz and semi-convex, the submanifolds themselves are $C^{1,1}$.  The solutions in \cite{NV,WdY} are $C^{1,\alpha}$, while the submanifolds  are analytic.  Our examples are not well defined in any sense except the viscosity one.

\section{Proof of the Hessian Estimates}\label{sec-Hess}
We prove Theorem \ref{main1} from which Corollary \ref{main0} follows.

\begin{proof} Let $n\geq 3$. To simplify the notation of the remaining proof, we take $R = 2n+2$ where $u$ is a solution on $B_{2n+2}\subset\re^n$. Then by scaling $v(x) = \frac{u(\frac{R}{2n+2}x)}{(\frac{R}{2n+2})^2}$, we get the estimate in Theorem \ref{main1}. We denote $C=C(n,\nu_1,\nu_2)(1 + \|Du\|^2_{L^\infty(B_{2n+1})})$ the positive constant from Lemma \ref{ptJ}. Let $b$ be defined as in Proposition \ref{IntJI}.

We observe that 
\begin{equation}\label{lowbd}
    b-\frac{1}{4}\ln(4/3)\geq \ln\sqrt{1 + \tan^2(\frac{\pi}{2}-\frac{\pi}{n})}-\frac{1}{4}\ln(4/3)\geq \frac{1}{4}\ln(4/3).
\end{equation}
Let $c_0 =\frac{1}{8}\ln(4/3)$, using \eqref{ptJI} and the above inequality we get
\begin{equation}
    \Delta_g (b-2c_0) \geq c(n)|\nabla_g b|^2 - C\geq -C\frac{(b-2c_0)}{2c_0}.
\end{equation}
\begin{enumerate}
    \item[Step 1.] We first show that $(b-2c_0)^\frac{n}{n-2}$ meets the requirements of the MVI \cite[Prop 5.1]{AB}:
    \begin{align*}
        &\int -\langle \nabla_g \phi, \nabla_g (b-2c_0)^\frac{n}{n-2}\rangle_g dv_g \\&= \int-\langle\nabla_g(\frac{n}{n-2}(b-2c_0)^\frac{2}{n-2}\phi) - \frac{2n}{(n-2)^2}(b-2c_0)^\frac{4-n}{n-2}\phi\nabla_g b,\nabla_g b\rangle_g dv_g\\
        &\geq \int (\frac{n}{n-2}c(n)(b-2c_0)^\frac{2}{n-2}\phi|\nabla_g b|^2)dv_g \\
        &\qquad + \int(\frac{2n}{(n-2)^2}(b-2c_0)^\frac{4-n}{n-2}\phi|\nabla_g b|^2)dv_g - \int C\frac{n}{n-2}\phi(b-2c_0)^\frac{2}{n-2} dv_g\\
        &\geq - \int C\frac{n}{n-2}\frac{1}{b-2c_0}(b-2c_0)^\frac{n}{n-2}\phi dv_g\\
        &\geq - \frac{C}{2c_0}\frac{n}{n-2}\int (b-2c_0)^\frac{n}{n-2}\phi dv_g
    \end{align*}
    where the last inequality comes from \eqref{lowbd}.

    Applying the MVI to the function $(b-2c_0)^\frac{n}{n-2}$ we get
    \begin{align}
        (b-2c_0)(0)&\leq C(n,|\Vec{H}|,c_0)\left(\int_{\bar B_1\cap L} (b-2c_0)^\frac{n}{n-2}\right)^\frac{n-2}{n}\nonumber\\
        &\leq C(n,|\Vec{H}|,c_0)\left(\int_{B_1} (b-2c_0)^\frac{n}{n-2}\right)^\frac{n-2}{n}\label{bMVI}
    \end{align}
    where $L=(x,Du(x))\subset\re^n\times\re^n$ is the Lagrangian submanifold, $\bar B_1$ is the ball of radius $1$ centered at $(0, Du(0))$ in $\re^n\times \re^n$, and $B_1$ is the ball of radius $1$ centered at $0$ in $\re^n$. Using the mean curvature formula \eqref{mc}, it follows that
    \[
    |\Vec{H}|^2=\sum_{j=1}^n\frac{1}{1+\lambda_j^2}(\Theta_{x_j} + \Theta_z u_j + \Theta_{p_j}\lambda_j)^2\leq C(n,\nu_1)\|Du\|^2_{L^\infty(B_1)}.
    \]
    
    We choose a cut off $\phi\in C_0^\infty(B_2)$ such that $0\leq\phi\leq 1$, $\phi = 1$ on $B_1$, and $|D\phi|< 2$. We get
    \begin{equation}\label{bcuttoff}
        \left(\int_{B_1}(b-2c_0)^\frac{n}{n-2}\right)^\frac{n-2}{n}\leq \left(\int_{B_2}\phi^\frac{2n}{n-2}(b-2c_0)^\frac{n}{n-2}\right)^\frac{n-2}{n} = \left( \int_{B_2}(\phi\sqrt{b-2c_0})^\frac{2n}{n-2}\right)^\frac{n-2}{n}.
    \end{equation}
    Assuming $\phi\sqrt{b-2c_0}$ is $C^1$ by approximation, we apply the Sobolev inequality \cite[Theorem 2.1]{MS} and use the mean curvature formula \eqref{mc} to get
    \begin{equation}\label{bSobolev}
        \left( \int_{B_2}(\phi\sqrt{b-2c_0})^\frac{2n}{n-2}\right)^\frac{n-2}{n}\leq C(n)[\int_{B_2}|\nabla_g (\phi\sqrt{b-2c_0})|^2dv_g + \int_{B_2}|\phi\sqrt{b-2c_0}\nabla_g\Theta|^2dv_g].
    \end{equation}
    To bound the left integral, we use the fact that
    \begin{align}
        |\nabla_g(\phi\sqrt{b-2c_0})|^2 &= |\frac{1}{2\sqrt{b-2c_0}}\phi\nabla_g b + \sqrt{b-2c_0}\nabla_g\phi|^2\nonumber\\
        &\leq \frac{1}{4(b-2c_0)}\phi^2|\nabla_g b|^2 + 2(b-2c_0)|\nabla_g \phi|^2\nonumber\\
        &\leq \frac{\phi^2}{8c_0}|\nabla_g b|^2 + 2(b-2c_0)|\nabla_g\phi|^2\nonumber\\
        &\leq \frac{\phi^2}{8c_0}|\nabla_g b|^2 + 2b|\nabla_g\phi|^2.\label{nablaCS}
    \end{align}

    Combining \eqref{bMVI}, \eqref{bcuttoff},\eqref{bSobolev}, and \eqref{nablaCS}, we get
    \begin{equation}\label{bThreeint}
        (b-2c_0)(0)\leq C(n,\nu_1,c_0)[\int_{B_2}\phi^2|\nabla_g b|^2dv_g + \int_{B_2}b|\nabla_g\phi|^2dv_g + \int_{B_2}\phi^2(b-2c_0)|\nabla_g\Theta|^2dv_g].
    \end{equation}

    \item[Step 2.] 
    To bound the third integral, we use \eqref{dipsi} to see that
    \begin{align}
        \int_{B_2}\phi^2(b-2c_0)|\nabla_g\Theta|^2dv_g&\leq C(\nu_1)\int_{B_2}\phi^2(b-2c_0)\sum_{j=1}^n\frac{1}{1 + \lambda_j^2}(1 + u_j^2 + \lambda_j^2)dv_g\nonumber\\
        &\leq C(\nu_1)(1 + \|Du\|^2_{L^\infty(B_2)})\int_{B_2}\phi^2b\sum_{j=1}^n\frac{1}{1 + \lambda_j^2}dv_g\nonumber \\&\qquad+ C(\nu_1)\int_{B_2}\phi^2(b-2c_0)\sum_{j=1}^n\frac{\lambda_j^2}{1 + \lambda_j^2}dv_g.\label{dtheta}
    \end{align}

    We seek to bound this last integral. Observe, 
    \begin{align*}
        -\int b\phi^2 u_j\Delta_g u_jdv_g &= \int \langle\nabla_g b\phi^2 u_j, \nabla_g u_j\rangle dv_g\\
        &= \int \phi^2 u_j\langle \nabla_g b, \nabla_g u_j\rangle + \int 2\phi b u_j\langle\nabla_g \phi,\nabla_g u_j \rangle dv_g + \int b \phi^2|\nabla_g u_j|^2 dv_g\\
        &\geq \int \frac{1}{2}(b-c_0)\phi^2|\nabla_g u_j|^2dv_g - \int\frac{1}{2c_0}\phi^2u_j^2|\nabla_g b|^2dv_g - \int 2u_j^2|\nabla_g \phi^2|dv_g
    \end{align*}
    where the last inequality comes from Young's inequality. Rearranging, we obtain
    \begin{equation}\label{bminusc}
        \int \frac{1}{2}(b-c_0)\phi^2|\nabla_g u_j|^2dv_g\leq -\int b\phi^2 u_j\Delta_g u_jdv_g+ \int\frac{1}{2c_0}\phi^2u_j^2|\nabla_g b|^2dv_g + \int 2u_j^2|\nabla_g \phi^2|dv_g.
    \end{equation}
    As $\Delta_g(x,Du(x)) = \Vec{H}$, we get that
    \[
    \Delta_g u_j = \langle\Vec{H},E_{n+j}\rangle = \langle J\nabla_g\Theta, E_{n+j}\rangle = \langle\nabla_g\Theta, E_j\rangle = g^{jk}\pd_k\Theta
    \]
    where $\{E_i\}$ is the unit orthonormal basis of $\re^{2n}$.
    Hence,
    \begin{align}
        -\int b \phi^2 &u_j\Delta_g u_j dv_g = -\int b\phi^2u_jg^{jk}\pd_k\Theta dv_g\nonumber\\
        &= -\int b\phi^2 u_j\frac{1}{1 + \lambda_j^2}(\Theta_x + \Theta_u u_j + \Theta_{u_j}\lambda_j)dv_g\nonumber\\
        &\leq C(\nu_1)(1 + \|Du\|^2_{L^\infty})\int b\phi^2\frac{1}{1 + \lambda_j^2}dv_g + C(\nu_1)\|Du\|_{L^\infty}\int b\phi^2 \frac{\lambda_j}{1 + \lambda_j^2}dv_g\nonumber\\
        &\leq C(\nu_1)(1 + \|Du\|^2_{L^\infty})\int b\phi^2\frac{1}{1 + \lambda_j^2}dv_g +  \frac{1}{4}\int b\phi^2\frac{\lambda_j^2}{1 + \lambda_j^2}dv_g
        \label{Deltauj}
    \end{align}
    where the last line comes from Young's inequality. 
    Using the fact that
    \begin{equation}\label{graduj}
        |\nabla_g u_j|^2 = \frac{\lambda_j^2}{1 + \lambda_j^2}
    \end{equation}
    and combining \eqref{Deltauj} and \eqref{bminusc} into \eqref{dtheta}, as well as \eqref{lowbd}, we get
    \begin{align*}
        \int_{B_2}\phi^2(b-2c_0)|\nabla_g\Theta|^2dv_g&\leq C(\nu_1)(1 + \|Du\|^2_{L^\infty(B_2)})\int_{B_2}\phi^2b\sum_{j=1}^n\left(\frac{1}{1 + \lambda_j^2}\right)dv_g\\
        &\quad + \frac{C(\nu_1)}{c_0}\|Du\|^2_{L^\infty}\left[\int\phi^2|\nabla_g b|^2dv_g + \int b|\nabla_g \phi^2|dv_g\right]
    \end{align*}
    \item[Step 3.] Plugging this into \eqref{bThreeint}, it follows
    \begin{align}
        (b-2c_0)(0) \leq C(n,\nu_1,c_0)(1 +& \|Du\|^2_{L^\infty})\bigg[\int_{B_2}\phi^2|\nabla_g b|^2dv_g + \int_{B_2}b|\nabla_g\phi|^2dv_g\label{bgrad1}\\
        &+\int_{B_2}\phi^2b\sum_{j=1}^n\frac{1}{1 + \lambda_j^2} dv_g\bigg].\nonumber
    \end{align}
    
    Using the integral Jacobi inequality, with the constants as in \eqref{IntJI}, we get
    \begin{align*}
        \int_{B_2}\phi^2|\nabla_g b|^2 dv_g &\leq \frac{1}{c(n)}[\int_{B_2}\phi^2\Delta_g b dv_g + \int_{B_2}\phi^2Cdv_g]\\
        &= \frac{1}{c(n)}[-\int_{B_2}2\phi\langle\nabla_g \phi,\nabla_g b\rangle dv_g + \int_{B_2}\phi^2Cdv_g]\\
        &\leq \frac{1}{2}\int_{B_2}\phi^2|\nabla_g b|^2dv_g + \frac{2}{c(n)^2}\int_{B_2}|\nabla_g \phi|^2dv_g + \frac{1}{c(n)}\int_{B_2}\phi^2 Cdv_g
    \end{align*}
    from which we conclude
    \begin{equation}\label{gradb}
        \int_{B_2}\phi^2|\nabla_g b|^2 dv_g \leq \frac{4}{c(n)^2}\int_{B_2}|\nabla_g \phi|^2dv_g + \frac{2}{c(n)}\int_{B_2}\phi^2 Cdv_g.
    \end{equation}
    Using \eqref{lowbd}, we get
    \begin{equation}\label{gradphi}
        \int_{B_2}|\nabla_g \phi|^2dv_g\leq\frac{1}{c_0}\int_{B_2}b|\nabla_g \phi|^2 dv_g \leq \frac{2}{c_0}\int_{B_2}b\sum_{j=1}^n\frac{1}{1 + \lambda_j^2}dv_g.
    \end{equation}
    Furthermore, we get that
    \begin{equation}\label{volbd}
        \int_{B_2}\phi^2 dv_g = \frac{1}{n}\int_{B_2}\phi^2\sum_{j=1}^n\left(\frac{1}{1 + \lambda_j^2} + \frac{\lambda_j^2}{1 + \lambda_j^2}\right) dv_g
    \end{equation}
    where we then bound the right-most integral using
    \begin{align*}
         \int \phi^2&\sum_{j=1}^n|\nabla_g u_j|^2dv_g\leq 4\int u_j^2|\nabla_g \phi|^2dv_g -2\int\phi^2 u_jg^{ij}\pd_j\Theta dv_g\\
         &\leq 16\|Du\|^2_{L^\infty}\int \sum_{j=1}^n\frac{1}{1 + \lambda_j^2}dv_g + C(\nu_1)(1 + \|Du\|^2_{L^\infty})\int\phi^2 \sum_{j=1}^n\frac{1}{1 +\lambda_j^2}dv_g\\
         &\qquad+ C(\nu_1)\|Du\|_{L^\infty}\int \phi^2\sum_{j=1}^n\frac{\lambda_j}{1 + \lambda_j^2} dv_g\\
         &\leq C(\nu_1)(1 + \|Du\|^2_{L^\infty})\int \sum_{j=1}^n\frac{1}{1 + \lambda_j^2}dv_g + \frac{1}{2}\int \phi^2\sum_{j=1}^n\frac{\lambda_j^2}{1+\lambda_j^2}dv_g.
    \end{align*}
   Hence, again using \eqref{graduj}, \eqref{volbd} is bounded by
    \begin{equation}\label{volbd2}
        \int_{B_2}\phi^2dv_g \leq C(n,\nu_1)(1  + \|Du\|^2_{L^\infty})\left[\int_{B_2}\sum_{j=1}^n\frac{1}{1 + \lambda_j^2}dv_g\right].
    \end{equation}
    Plugging \eqref{volbd2} and \eqref{gradphi} into \eqref{gradb}, plugging \eqref{gradb} into \eqref{bgrad1}, and using \eqref{lowbd}, we obtain
    \begin{equation*}
        (b-2c_0)(0) \leq C(n,\nu_1,\nu_2,c_0)(1 + \|Du\|^2_{L^\infty})\left[\int_{B_2}b\sum_{j=1}^n\frac{1}{1 + \lambda_j^2}dv_g\right].
    \end{equation*}
    \item[Step 4.] We will bound
    \[
    \int_{B_2}b\sum_{j=1}^n\frac{1}{1 + \lambda_j^2}dv_g
    \]
    as in \cite{WaY}. We use the conformality identity:
    \begin{equation*}
        \left(\frac{1}{1+\lambda_1^2},\dots,\frac{1}{1 + \lambda_n^2}\right)V = \cos\Theta D_\lambda(\sigma_1 - \sigma_3 + \cdots) - \sin\Theta D_\lambda(\sigma_0 - \sigma_2 + \cdots).
    \end{equation*}
    Taking the trace gives
    \begin{equation}\label{conform}
        \sum_{j=1}^n\frac{1}{1 + \lambda_j^2}V = \cos\Theta\sum_{0\leq 2k< n}(-1)^k(n-2k)\sigma_{2k} - \sin\Theta\sum_{0\leq 2k-1< n}(-1)^k(n-2k+1)\sigma_{2k-1}.
    \end{equation}
    Each coefficient of $\sigma_k$ in \eqref{conform} is bounded above by $n$ since $|\cos\Theta|,|\sin\Theta|\leq 1$, so it suffices to bound
    \[
    \int_{B_2} b \sigma_k dx
    \]
    for all $0\leq k \leq n-1$. We use the divergence structure of $\sigma_k(D^2u)$:
    \begin{align*}
        k\sigma_k(D^2u) &= \sum_{i,j=1}^n\frac{\pd \sigma_k}{\pd u_{ij}}\frac{\pd^2 u}{\pd x_i\pd x_j} = \sum_{i,j=1}^n\frac{\pd}{\pd x_i}\left(\frac{\pd \sigma_k}{\pd u_{ij}}\frac{\pd u}{\pd x_j}\right)\\
        &= \text{div}(L_{\sigma_k} Du)
    \end{align*}
    where $L_{\sigma_k}$ is the matrix $(\frac{\pd \sigma_k}{\pd u_{ij}})$. As $\sigma_k>0$ for each $0\leq k\leq n-1$, we use a cuttoff function $\phi$ such that $0\leq \phi\leq 1$, $\phi=1$ on $B_2$, and $|D\phi|<2$. We get
    \begin{align}
        \int_{B_2} b \sigma_kdx&\leq \int_{B_3}\phi b\sigma_kdx = \int_{B_3}\phi b \frac{1}{k}\text{div}(L_{\sigma_k}Du)dx\nonumber\\
        &= \frac{1}{k}\int_{B_3}-\langle bD\phi + \phi Db,L_{\sigma_k}Du\rangle dx\nonumber\\
            &\leq C(n)\|Du\|_{L^\infty}\bigg[\int_{B_3} b\sigma_{k-1}\nonumber \\&\hspace{1.5in}+ \int_{B_3}[|\nabla_g b|^2 + \text{tr}(g^{ij})]\sqrt{\det g}dx \bigg]\label{bsigk}
    \end{align}
    where the last inequality comes from the argument in \cite{WaY}.
    
    Using \eqref{gradb} and \eqref{volbd2}, it follows that
    \begin{equation}\label{gradb2}
    \int_{B_r} \phi^2|\nabla_g b|^2dv_g \leq C(n,\nu_1,\nu_2)(1 + \|Du\|^2_{L^\infty})\int_{B_{r+1}} \text{tr}(g^{ij})\sqrt{\det g}dx.
    \end{equation}
    Combining \eqref{gradb2} and \eqref{bsigk}, we get the inductive inequality
    \[
    \int_{B_r}b\sigma_kdx\leq C(n,\nu_1,\nu_2)\|Du\|_{L^\infty}\left[\int_{B_{r+1}} b\sigma_{k-1}dx + (1 + \|Du\|^2_{L^\infty})\int_{B_{r+2}}\text{tr}(g^{ij})\sqrt{\det g}dx \right].
    \]
    Using this induction for each $\sigma_k$, with $0\leq k \leq n-1$, altogether we get
    \begin{align}
        (b-2c_0)(0)&\leq C(n,\nu_1,\nu_2)(1+\|Du\|^2_{L^\infty(B_{r+n+2})})\Bigg(\sum_{j=1}^n\|Du\|_{L^\infty(B_{r+n+2})}^j\nonumber \\
        &\qquad + \left(\sum_{j=1}^{n+2}\|Du\|_{L^\infty(B_{r+n+2})}^j\right)\int_{B_{r+n+2}}\text{tr}(g^{ij})\sqrt{\det g}dx\Bigg)\label{induct1}
    \end{align}
    where we use that
    \[
    \int_{B_r} b dx\leq C(n)\|Du\|_{L^\infty(B_r)}
    \]
    as in \cite{WaY}.
    
    In order to bound the remaining integral, we use the conformality identity \eqref{conform} to get
    \begin{equation}\label{trace}
        \int_{B_{r}}\text{tr}(g^{ij})\sqrt{\det g}dx \leq C(n)\sum_{j=0}^{n-1}\int_{B_{r+1}}\sigma_jdx.
    \end{equation}
    The argument above with $b=1$ gives the inductive inequality
    \[
    \int_{B_r}\sigma_{k}dx \leq C(n)\|Du\|_{L^\infty(B_{r+1})}\int_{B_{r+1}}\sigma_{k-1}dx,
    \]
    which gives
    \begin{equation*}
        \int_{B_{r}}\text{tr}(g^{ij})\sqrt{\det g}dx\leq C(n)\sum_{j=0}^{n-1}\|Du\|^{j}_{L^\infty(B_{r+n})}.
    \end{equation*}
    Plugging these into \eqref{induct1}, and using Young's inequality, we get
    \[
    (b-2c_0)(0) \leq C(n,\nu_1,\nu_2)(1+\|Du\|^2_{L^\infty(B_{2n+1})})\left[\|Du\|^{2n+1}_{L^\infty(B_{2n+1})}+1 \right].
    \]
    Adding $2c_0$ across and exponentiating yields
    \[
    |D^2u(0)| \leq C_1\exp{[C_2\|Du\|^{2n+3}_{L^\infty(B_{2n+1})}]}
    \]
    where $C_1$ and $C_2$ depend on $n,\nu_1$, and $\nu_2$.
    \end{enumerate}
\end{proof}

\section{$C^{2,\alpha}$ Regularity} \label{sec-c2alpha}
Once Hessian and gradient estimates are established, a final challenge remains in proving $C^{2,\alpha}$ estimates. For general equations $F(D^2u, Du, u, x) = 0$, higher regularity results require both concavity and uniform ellipticity \cite{Krylov}. Without concavity, known results apply to uniformly elliptic equations of the form $F(D^2u) = f(x)$ with interior estimates \cite{CCannals} or convex level sets of the graph \cite{caffarelli2000priori}. 

Here we will prove that the arctangent operator can be modified to a concave operator in dimensions $n>2$.

\begin{lemma}\label{Expo}
    Let $f(\lambda) = \sum_{j=1}^n \arctan(\lambda_j)$ be defined on $\{\lambda\in\re^n |\; F(D^2u)= \sum_{j=1}^n\arctan\lambda_j \geq (n-2)\frac{\pi}{2}\}$ and suppose $u$ is a solution of \eqref{slag} on $B_R$, then there exists a constant $A$ depending on $\|D^2u\|_{L^\infty(B_R)}$ and $n>2$ such that $-e^{-Af(\lambda)}$ is a concave function.
\end{lemma}
\begin{remark}
    Note this exponentiation technique does not work in two dimensions due to counterexamples of Yuan \cite{Ycount_notes}.
\end{remark}

\begin{proof}
    Using the calculation as in \cite[Lemma 2.2]{CPW}, observe
    \begin{equation}\label{dex}
        e^{Af}\pd^2_{ij} e^{-Af} =-A^2f_i f_j - Af_{ij} = -A \left(\frac{A + 2\lambda_i\delta_{ij}}{(1 + \lambda_i^2)(1+\lambda_j^2)}\right) := -AH_{ij}
    \end{equation}
    where
    \[
    \det H_{ij} = \frac{1}{\prod(1 + \lambda_i^2)^2}(A2^{n-1} \sigma_{n-1}(\lambda) + 2^n \sigma_n(\lambda)).
    \]
    If $\lambda_n \geq 0$, then we are done, so suppose that $\lambda_n < 0$. We show that $\sigma_{n-1} \geq C > 0$. Let $\theta_j = \arctan(\lambda_j)$ and observe that
    \[
    \frac{(n-2)\pi}{2}- \theta_n \leq \sum_{j=1}^{n-1}\theta_j \leq (n-2)\arctan(\|D^2u\|_{L^\infty(B_R)}) + \theta_{n-1}.
    \]
    Rearranging the above, noting that $-\theta_n > 0$, and taking the tangent of both sides, we obtain
    \begin{equation}\label{lnlower}
        \lambda_{n-1} > \tan(T)
    \end{equation}
    where $T =\frac{\pi}{2}- \arctan(\|D^2u\|_{L^\infty(B_R)})$. Note that $n>2$ is crucial here.
    
    Fix
    \begin{equation}\label{epsis}
        \epsilon = \frac{s}{(n-1)} \quad\text{where}\quad s= \frac{\tan(T)}{2},
    \end{equation}
    and suppose that $0>\lambda_n \geq -\epsilon $. We denote 
    \[
    \hat{\lambda}_j = \prod_{i\neq j}^n \lambda_i \quad\text{and}\quad \widehat{\lambda_i\lambda_j} = \prod_{k\neq i,j}^n\lambda_k.
    \]
    We then observe
    \begin{align*}
        \sigma_{n-1}(\lambda)= \sum_{j=1}^n\hat{\lambda}_j &\ge \hat{\lambda}_n - s\sum_{j=1}^{n-1}\widehat{\lambda_j\lambda_n}\\
        & = \frac{\sum_{j=1}^{n-1}\left[(\lambda_j - s)\widehat{\lambda_j\lambda_n}\right]}{n-1}\\
        &\geq \frac{(n-1)(\lambda_{n-1} - s)(\lambda_{n-1})^{n-2}}{n-1}> \frac{1}{2}[\tan(T)]^{n-1} = C
    \end{align*}
    where the last inequality comes from \eqref{lnlower} and \eqref{epsis}.

    On the other hand, if $ \lambda_n <- \epsilon$, we get
    \[
    -\frac{2(n-1)}{\tan(T)} > \frac{1}{\lambda_n},
    \]
    which combined with
    \[
    \frac{(n-1)}{\tan(T)} > \frac{(n-1)}{\lambda_{n-1}} \geq \sum_{j=1}^{n-1}\frac{1}{\lambda_j}
    \]
    gives
    \[
    0 > -\frac{(n-1)}{\tan(T)} > \sum_{j=1}^n \frac{1}{\lambda_n} = \frac{\sigma_{n-1}(\lambda)}{\sigma_{n}(\lambda)}.
    \]
    Multiplying across by $\sigma_n(\lambda) < 0$ flips the inequalities and we get
    \[
    0 < \frac{(n-1)(-\sigma_n(\lambda))}{\tan(T)} < \sigma_{n-1}(\lambda).
    \]
    Noting that 
    \[
    -\sigma_n(\lambda_n) \geq |\lambda_n|^n > \epsilon^n = \frac{[\tan(T)]^{n}}{2^n(n-1)^n} 
    \]
    gives $\sigma_{n-1}(\lambda) \geq C > 0$. 

    Now, choose $A$ large enough so that
    \begin{align*}
        \det H_{ij} &= \frac{1}{\prod(1 + \lambda_i^2)^2}(A2^n \sigma_{n-1}(\lambda) + 2^n \sigma_n(\lambda)) \\
        & \geq \frac{1}{\prod(1 + \lambda_i^2)^2}(A2^{n-1} C - 2^n \|D^2u\|_{L^\infty(B_R)}^n) > 0,
    \end{align*}
    that is,
    \[
    A > \frac{2\|D^2u\|_{L^\infty(B_R)}}{C},
    \]
    and this completes the proof.
\end{proof}

\begin{proposition}\label{c2alph1}
        Let $u$ be a $C^{1,1}$ solution of \eqref{slag} on $B_1\subset\re^n$. Then $u\in C^{2,\alpha}(B_{1/2})$ with
        \[
        |D^2u(x) - D^2u(0)| \leq C(\|u\|_{C^2(B_1)},\nu_1,\nu_2,n)|x|^\alpha.
        \]
\end{proposition}

\begin{corollary}
    Let $u$ be a $C^{1,1}$ solution of \eqref{s}, \eqref{tran}, or \eqref{rotator} on $B_1\subset\re^n$. Then $u\in C^{2,\alpha}(B_{1/2})$ with the same estimate above where $C$ also depends on $s_2$ for \eqref{s}, $\gamma_2,\gamma_3$ for  \eqref{tran}, or $r_2$ for \eqref{rotator}.
\end{corollary}

\begin{proof}
    We modify the second order, uniformly elliptic equation \eqref{slag} to a concave one via exponentiation from Lemma \ref{Expo}. Applying \cite[Cor. 1.3]{CC1}, we get $u\in C^{2,\alpha}(B_{1/2})$ with the estimate
    \[
    |D^2u(x) - D^2u(0)|\leq C(\|u\|_{L^\infty(B_1)}+\|e^{-A\Theta}\|_{C^\alpha(B_1)})|x|^\alpha
    \]
    where $C$ depends on the uniform ellipticity constants and the dimension $n$. 
    
    By Taylor expanding the exponential, we get
    \[
    \|e^{-A\Theta}\|_{C^\alpha(B_{1})}\leq A\|\Theta\|_{C^\alpha(B_{1})}.
    \]
    To bound $\|\Theta\|_{C^\alpha(B_{1})}$, we Taylor expand $F(x) = \Theta(x,u(x), Du(x))$ to get 
    \[
    \Theta(x_1,u(x_1),Du(x_1)) - \Theta(x_2, u(x_2), Du(x_2)) = DF(\xi)\cdot (x_1 - x_2)
    \]
    for $\xi$ between $x_1,x_2$, where
    \begin{align}
        |DF(\xi)| &= |\Theta_x(\xi,u(\xi),Du(\xi)) + \Theta_z(\xi,u(\xi),Du(\xi))Du(\xi) + \Theta_{p}(\xi,u(\xi),Du(\xi))D^2u(\xi)|\nonumber\\
        &\leq \nu_1(1 + \|Du\|_{L^\infty(B_1)}+ \|D^2u\|_{L^\infty(B_1)})\label{alphgen}
    \end{align}
    which bounds $\|\Theta\|_{C^\alpha(B_1)}$.

    When $u$ solves the translator equation \eqref{tran}, we get that
    \begin{align}
        &\frac{|\Theta(x_1,u(x_1),Du(x_1)) -\Theta(x_2,u(x_2),Du(x_2))|}{|x_1 - x_2|^\alpha} = \frac{|\gamma_2\cdot(x_1 - x_2) + \gamma_3\cdot(Du(x_1) - Du(x_2))|}{|x_1 - x_2|^\alpha}\nonumber\\
        &\qquad\qquad\leq |\gamma_2| + |\gamma_3|\|Du\|_{C^\alpha(B_{1})}.\label{alphtran}
    \end{align}

    When $u$ solves the shrinker/expander equation \eqref{s}, we add and subtract $x_1\cdot Du(x_2)$ in the numerator to get
    \begin{align}
        &\frac{|\Theta(x_1,u(x_1),Du(x_1)) -\Theta(x_2,u(x_2),Du(x_2))|}{|x_1 - x_2|^\alpha}\nonumber\\
        &\qquad\qquad= |s_2|\frac{|x_1\cdot Du(x_1) - x_2\cdot Du(x_2) - 2(u(x_1) - u(x_2))|}{|x_1 - x_2|^\alpha}\nonumber\\
        &\qquad\qquad= |s_2|\frac{|x_1\cdot (Du(x_1) - Du(x_2)) + Du(x_2)\cdot(x_1 - x_2) - 2(u(x_1) - u(x_2))|}{|x_1 - x_2|^\alpha}\nonumber\\
        &\qquad\qquad\leq |s_2|\left(\|Du\|_{C^\alpha(B_{1})} + \|Du\|_{L^\infty(B_1)} + 2\|u\|_{C^\alpha(B_{1})}\right)\label{alphshrink}.
    \end{align}

    Lastly, when $u$ solves the rotator equation \eqref{rotator}, we use that
    \[
    ||v|^2 - |w|^2| = |\langle v + w , v - w\rangle| \leq |v+w||v-w|
    \]
    to get
    \begin{align}
        &\frac{|\Theta(x_1,u(x_1),Du(x_1)) -\Theta(x_2,u(x_2),Du(x_2))|}{|x_1 - x_2|^\alpha}\nonumber\\
        &\qquad\qquad= \frac{|r_2|}{2}\frac{| |x_1|^2 - |x_2|^2 + |Du(x_1)|^2 - |Du(x_2)|^2 |}{|x_1 - x_2|^\alpha}\nonumber\\
        &\qquad\qquad\leq \frac{|r_2|}{2}\left(2 + 2\|Du\|_{L^\infty(B_1)}\|Du\|_{C^\alpha(B_{1})}\right).\label{alphrot}
    \end{align}
    The estimates in \eqref{alphgen}, \eqref{alphtran}, \eqref{alphshrink}, and \eqref{alphrot} complete the proof.
\end{proof}

\subsection{Discussion: alternate approaches}

Here we outline an alternate proof using the upward rotation method of Bhattacharya-Ogden \cite[Lemma 4.1]{BOHamstat}, which is based on a rotation of Yuan \cite{YY0}. Note that this will work in dimension two.

\begin{lemma}[Lemma 4.1, \cite{BOHamstat}]
    Let $u \in C^{1,1}(B_R)$ for $B_R\subset \re^n$. Let 
    \[
    K = \sup_{B_R}|D^2u| \quad \text{and} \quad \gamma = \frac{1}{2}\left(\frac{\pi}{2} - \arctan K\right)
    \]
    Then the Lagrangian submanifold $L = (x, Du(x)) \subset\re^n\times\re^n$ can be represented as a gradient graph $(\bar{x},D\bar{u}(\bar{x}))$ given by 
    \[
        \begin{pmatrix}
            \bar{x}\\
            \bar{y}
        \end{pmatrix} = 
        \begin{pmatrix}
            \cos\gamma & -\sin\gamma\\
            \sin\gamma & \cos\gamma
        \end{pmatrix}
        \begin{pmatrix}
            x\\
            y
        \end{pmatrix}.
    \]
    The new potential is given by 
    \begin{equation}\label{rotpot}
        \bar{u}(x) = u(x) - \sin\gamma \cos\gamma\frac{|Du|^2 - |x|^2}{2}- \sin^2\gamma(x\cdot Du(x)).
    \end{equation}
    The rotated coordinates are Lipschitz with respect to the original coordinates by
    \[
    |\bar{x}_1 - \bar{x_2}|\geq (\cos\gamma - K\sin\gamma)|x_1 - x_2|
    \]
    The rotated gradient graph in the $\bar{x}$ coordinates is Lipschitz with constant
    \[
    |D\bar{u}(\bar{x_1}) - D\bar{u}(\bar{x}_2) |\geq\left(\frac{\sin\gamma + \|u\|_{C^{1,1}(B_R)}}{\cos\gamma - K\sin\gamma}\right)|\bar{x_1} - \bar{x_2}|.
    \]
\end{lemma}

Now we prove the $C^{2,\alpha}$ estimates.
\begin{proof}[Proof of the $C^{2,\alpha}$ estimates]
    We will prove these estimates on a neighborhood of 
    \[
    U_0 = \{x\in B_{\frac{1}{2}} |\; \Theta(x,u(x),Du(x)) \geq (n-2)\pi/2\}
    \] 
    as by symmetry they will hold for the set where $\Theta \leq -(n-2)\pi/2$. Let $\Omega = \bar{x}(U)$ for $U = \{x\in B_1 |\; \Theta(x,u(x),Du(x)) \geq (n-2)\frac{\pi}{2} - (n-1)\gamma\}\supset U_0 $. On $\Omega$, $\bar{u}$ satisfies 
    \begin{equation}\label{roteq}
        \sum_{j=1}^n \arctan\bar{\lambda}_j = \bar{\Theta}(\bar{x},\bar{u}(\bar{x}),D\bar{u}(\bar{x})) := g(\bar{x})    
    \end{equation}
    where the $\bar{\lambda}_j$ are the eigenvalues of $D^2\bar{u}$. As the upward rotation increases the phase, we get
    \[
    g(\bar{x}) = \sum_{j=1}^n \arctan\bar{\lambda}_j = \sum_{j=1}^n \arctan\lambda_j + n\gamma = \Theta(x,u(x),Du(x)) + n\gamma
    \]
    which shows that $g(\bar{x}) \geq (n-2)\pi/2 + \gamma$ in $\Omega$. We then have for $\bar{r} = (\cos\gamma - K\sin\gamma)\rho/2$ where $B_{\bar{r}}\subset\Omega$ and $B_{\rho}\subset U$, that $g\in C^\alpha(B_{\bar{r}})$ with
    \[
    \frac{|g(\bar{x}_1) - g(\bar{x}_2)|}{|\bar{x}_1 - \bar{x}_2|^\alpha}\leq \frac{|\Theta(x_1,u(x_1),Du(x_1) - \Theta(x_2,u(x_2),Du(x_2))|}{|x_1 - x_2|^\alpha}(\cos\gamma - K\sin\gamma)^{-\alpha}.
    \]

    We modify the uniformly elliptic second order equation \eqref{roteq} to a concave one via exponentiation (see \cite[p. 24]{CW2} or \cite[Lemma 2.2]{CPW}). Applying \cite[Cor. 1.3]{CC1}, we get $\bar{u}\in C^{2,\alpha}(B_{\bar{r}/2})$ with the estimate
    \[
    |D^2\bar{u}(\bar{x}) - D^2\bar{u}(0)|\leq C\left(\|\bar{u}\|_{L^\infty(B_{\bar{r}})} + \|e^{-Ag}\|_{C^\alpha(B_{\bar{r}})}\right)|\bar{x}|^\alpha
    \]
    where $C$ depends on $K$ and $n$, and $A$ depends on $\gamma$. We rotate the gradient graph $\{(\bar{x},D\bar{u}(\bar{x})) | \bar{x}\in B_{\bar{r}/2}\}$ back down by an angle of $\gamma$ to the original coordinates:
    \[
    \begin{pmatrix}
        x\\
        y
    \end{pmatrix} = 
    \begin{pmatrix}
        \cos \gamma & \sin \gamma \\
        -\sin\gamma & \cos \gamma
    \end{pmatrix}
    \begin{pmatrix}
        \bar{x}\\
        \bar{y}
    \end{pmatrix}.
    \]
    We then connect the original Hessian to the rotated Hessian using 
    \[
    \frac{dx}{d\bar{x}} = \cos\gamma I_n + \sin \gamma D^2_{\bar{x}}\bar{u}(\bar{x}) = \left(\frac{d\bar{x}}{dx}\right)^{-1}\geq \frac{1}{\cos\gamma + K\sin\gamma}I_n
    \]
    to get, for $x\in B_{r_0}\subset \bar{x}^{-1}(B_{\bar{r}/2})$,
    \[
    |D^2u(x) - D^2u(0)|\leq C\left(\|\bar{u}\|_{L^\infty(B_{\bar{r}})} + \|e^g\|_{C^\alpha(B_{\bar{r}}(0))}\right)(\cos\gamma + K\sin\gamma)^{2+\alpha}|x|^\alpha.
    \]
    We note that by Taylor expanding the exponential we get
    \[
    \|e^{-Ag}\|_{C^\alpha(B_{\bar{r}})}\leq A\|g\|_{C^\alpha(B_{\bar{r}})}.
    \]
    
    Using \eqref{rotpot}, connecting the $C^\alpha$ norms, and bounding $\|\Theta\|_{C^\alpha(B_\rho)}$ as in Proposition \ref{c2alph1}, we get
    \begin{align*}
        |D^2u(x) - D^2u(0)| &\leq C\bigg(\|u\|_{L^\infty(B_R)} + R\|Du\|_{L^\infty(B_R)}\\ 
        &\quad+ R^2/2 + \|Du\|^2_{L^\infty(B_R)}/2 + 2\lambda^\alpha \bar{r} C_\Theta \bigg)|x|^\alpha
    \end{align*}
    where $C_\Theta$ is given by the estimates in \eqref{alphgen}, \eqref{alphtran} \eqref{alphshrink}, or \eqref{alphrot}. \end{proof}
    
   \begin{remark}\label{C2a}
       We mention another approach. One can show that the second fundamental form of the level set is uniformly convex for $n>2$ following the proof in \cite{YY0} with strict inequalities. From there one can exponentiate as in \eqref{dex} where we choose a basis of vectors given by the $(n-1)$ tangential directions and one normal direction to the level set. As the $(n-1)\times(n-1)$ upper left block matrix is the second fundamental form, it is positive definite by the uniform convexity of the level set. Choosing $A$ large enough in the remaining lower right block matrix makes the whole matrix positive definite.
        \end{remark} 

\bibliographystyle{amsalpha}
\bibliography{main}

\end{document}